\documentclass[a4paper]{ieeeconf}

\usepackage{times,mathptmx}
\usepackage{latexsym,amsmath,amsfonts}

\IEEEoverridecommandlockouts

\newcommand{\R}{\mathbb{R}}
\newcommand{\T}{\mathbb{T}}

\newcommand{\Z}{\mathbb{Z}}

\newtheorem{ex}{Example}
\newtheorem{theorem}{Theorem}
\newtheorem{lemma}{Lemma}
\newtheorem{corollary}{Corollary}
\newtheorem{definition}{Definition}
\newtheorem{remark}{Remark}

\begin{document}

\title{\LARGE \bf
The Second Euler-Lagrange Equation\\
of Variational Calculus on Time Scales\thanks{Submitted 26-May-2009;
Revised 12-Jan-2010; Accepted 29-March-2010 in revised form;
for publication in the \emph{European Journal of Control}.}}

\author{Zbigniew Bartosiewicz, Nat\'{a}lia Martins,
and Delfim F. M. Torres%
\thanks{Corresponding author:
Nat\'{a}lia Martins ({\tt\small natalia@ua.pt})}%
\thanks{This work was partially presented at the
\emph{Workshop in Control, Nonsmooth Analysis and Optimization},
celebrating Francis Clarke's and Richard Vinter's 60th birthday,
Porto, May 4-8, 2009.}%
\thanks{Z. Bartosiewicz is with Faculty of Computer Science,
Bia{\l}ystok University of Technology, 15-351 Bia\l ystok, Poland.
{\tt\small z.bartosiewicz@pb.edu.pl}}%
\thanks{N. Martins is with the Department of Mathematics,
University of Aveiro, 3810-193 Aveiro, Portugal.
{\tt\small natalia@ua.pt}}%
\thanks{D. F. M. Torres is with the Department of Mathematics,
University of Aveiro, 3810-193 Aveiro, Portugal.
{\tt\small delfim@ua.pt}}}

\maketitle

%-----------------------------------------

\begin{abstract}
The fundamental problem of the calculus of variations
on time scales concerns the minimization of a delta-integral
over all trajectories satisfying given boundary conditions.
In this paper we prove the second Euler-Lagrange necessary
optimality condition for optimal trajectories
of variational problems on time scales.
As an example of application of the main result,
we give an alternative and simpler proof to the
Noether theorem on time scales recently obtained in
[J. Math. Anal. Appl. 342 (2008), no.~2, 1220--1226].
\end{abstract}

\smallskip


\smallskip

%-----------------------------------------

\smallskip


%-----------------------------------------

\section{INTRODUCTION}

The calculus on time scales is a recent field,
introduced by Bernd Aulbach and Stefan Hilger in 1988
\cite{Hilger}, that unifies and extends difference and differential equations
into a single theory \cite{bookBohner}. A time scale is a model of time,
and the new theory has found important applications in several fields
that require simultaneous modeling of discrete and continuous data,
in particular in the calculus of variations
\cite{Ahlbrandt,comRicardoISO:nabla,Atici,BohnerCV04,Tangier-with-Rui,%
comRui:TS:Lisboa07,CVts,abmalina:t,Mal:Tor:Wei,NataliaHigherOrderNabla},
control theory \cite{Barto:et:al,BartoPaw,BartoPaw:IJMS,BPW,D1,MozBart},
and optimal control \cite{CV:WMP,SSW,ZZ:WW:HX}.
Other areas of application include
engineering, biology, economics, finance,
and physics \cite{ts:survey,bookBohner}.

The present work is dedicated
to the study of problems of calculus of variations on an arbitrary
time scale $\mathbb{T}$. As particular cases,
one gets the standard calculus of variations
\cite{Clar83,book:Vinter} by choosing $\mathbb{T} = \mathbb{R}$;
the discrete-time calculus of variations \cite{KP,discreteCAO03}
by choosing $\mathbb{T} = \mathbb{Z}$;
and the $q$-calculus of variations
\cite{Bang:q-calc} by choosing $\mathbb{T} =
q^{\mathbb{N}_0}:=\{q^k | k \in \mathbb{N}_0\}$, $q>1$.
In Section~\ref{sec:prel:ts} we briefly present the necessary notions
and results of time scales, delta derivatives, and delta integrals.

Let $\mathbb{T}$ be a given time
scale with at least three points, $n \in \mathbb{N}$,
and $L: \mathbb{R}\times \mathbb{R}^n \times \mathbb{R}^n
\rightarrow \mathbb{R}$ be of class $C^1$.
Suppose that $a,b\in \mathbb{T}$ and $a<b$.
We consider the following optimization problem on $\mathbb{T}$:
\begin{equation}
\label{problem}
I[q]=\int_a^b L(t,q^\sigma(t),q^\Delta(t)) \Delta t
\longrightarrow \min_{q\in \mathcal{D}},
\end{equation}
where
\[
\mathcal{D}=\{ q\ | \ q: [a,b]\cap \mathbb{T}
\rightarrow \mathbb{R}^n,\ q\in \mathrm{C}^1_{rd},\
q(a)=q_a,\ q(b)=q_b\}
\]
for some $q_a, q_b \in \mathbb{R}^n$, and
where $\sigma$ is the forward jump operator and $q^\Delta$
is the delta-derivative of $q$ with respect to $\mathbb{T}$.
For $\mathbb{T} = \mathbb{R}$ we get the classical fundamental
problem of the calculus of variations,
which concerns the minimization of an integral
\begin{equation}
\label{eq:cp}
I[q]=\int_a^b L(t,q(t),\dot{q}(t)) dt
\end{equation}
over all trajectories $q\in \mathrm{C}^1$
satisfying given boundary conditions
$q(a)=q_a$ and $q(b)=q_b$.
Several classical results on the calculus of variations
are now available to the more general context of time scales:
(first) Euler-Lagrange equations \cite{Atici,BohnerCV04,CVts};
necessary optimality conditions for isoperimetric problems
\cite{comRicardoISO:nabla,abmalina:t} and for problems
with higher-order derivatives
\cite{comRui:TS:Lisboa07,NataliaHigherOrderNabla};
the Weierstrass necessary condition \cite{Mal:Tor:Wei};
and Noether's symmetry theorem \cite{NT:ts}.
In this paper we prove a new result for the problem
of the calculus of variations on time scales:
we obtain in Section~\ref{sec:MR} a time scale version of the
classical \emph{second Euler-Lagrange equation}
\cite{book:Troutman}, also known in the literature as the
DuBois-Reymond necessary optimality condition \cite{Cesari}.

The classical second Euler-Lagrange equation asserts that if $q$
is a minimizer of \eqref{eq:cp}, then
\begin{multline}
\label{sec:class:EL}
\frac{d}{dt}\left[-L(t,q(t), \dot{q}(t))
+ \partial_3L(t,q(t), \dot{q}(t))\cdot \dot{q}(t)\right] \\
= - \partial_1 L (t,q(t), \dot{q}(t))\, ,
\end{multline}
where $\partial_i L$, $i = 1, 2, 3$, denotes the
partial derivative of $L(\cdot,\cdot,\cdot)$
with respect to its $i$-th argument. In the autonomous case,
when the Lagrangian $L$ does not depend on the time variable $t$,
the second Euler-Lagrange condition \eqref{sec:class:EL}
is nothing more than the second Erdmann necessary optimality condition:
\begin{equation}
\label{eq:cla:2nd:Erd}
- L(q(t), \dot{q}(t))
+ \partial_3L(q(t), \dot{q}(t))\cdot \dot{q}(t) = const
\end{equation}
along all the extremals of the problem,
which in mechanics corresponds to the most famous
conservation law---conservation of energy.
For a survey of the classical optimality conditions
we refer the reader to \cite[Ch.~2]{Clar89}.
Here we just recall that \eqref{sec:class:EL}
is one of the cornerstone results of the calculus of variations
and optimal control \cite{Caratheodory}:
it has been used, for example, to prove existence,
regularity of minimizers, conservation laws,
and to explain the Lavrentiev phenomena.

Main result of the paper gives an extension of \eqref{sec:class:EL}
to an arbitrary time scale (\textrm{cf.} Theorem~\ref{secondEL}):
if $q$ is a solution of problem \eqref{problem}, then
\begin{multline}
\label{eq:new:DRC}
\frac{\Delta}{\Delta t} \Bigl[
-L(t,q^\sigma(t),q^\Delta(t))
+ \partial_3 L(t,q^\sigma(t),q^\Delta(t)) q^\Delta(t)\\
+ \partial_1 L(t,q^\sigma(t),q^\Delta(t)) \mu(t)
\Bigr] = - \partial_1 L(t,q^\sigma(t),q^\Delta(t)) \, .
\end{multline}
As an application, we show in Section~\ref{sec:appl:NT}
how one can use the new second Euler-Lagrange equation
\eqref{eq:new:DRC} to prove the Noether's principle
on time scales: the invariance of functional \eqref{problem}
with respect to a one-parameter family of transformations
implies the existence of a conserved
quantity along the time scale Euler-Lagrange extremals
(Theorem~\ref{Noether}). When problem \eqref{problem}
is autonomous one has invariance with respect to time translations
and the corresponding Noether's conservation law gives
an extension of the second Erdmann equation
\eqref{eq:cla:2nd:Erd} to time scales:
\begin{equation}
\label{eq:cons:Eng}
-L(q^\sigma(t),q^\Delta(t))
+ \partial_3 L(q^\sigma(t),q^\Delta(t)) \cdot q^\Delta(t) = const \, .
\end{equation}
In classical mechanics \eqref{eq:cla:2nd:Erd}
gives conservation of energy. The conservation law
\eqref{eq:cons:Eng} tells us that an analogous result
remains valid on an arbitrary time scale.
However, the role of the classical
Hamiltonian $H(t,u,v)=-L(t,u,v) + \partial_3 L(t,u,v)v$
in \eqref{sec:class:EL} is substituted
by a time-scale Hamiltonian
$$
\mathcal{H}(t,u,v)=-L(t,u,v)
+ \partial_3 L(t,u,v)v + \partial_1 L(t,u,v) \mu(t)
$$
in \eqref{eq:new:DRC}, \textrm{i.e.},
$\mathcal{H}(t,u,v)=H(t,u,v) + \partial_1 L(t,u,v) \mu(t)$,
with the new term
$\partial_1 L(t,q^\sigma(t),q^\Delta(t)) \mu(t)$
on the left-hand side of \eqref{eq:new:DRC}
depending on the graininess of the time scale.

%-----------------------------------------

\section{PRELIMINARIES ON TIME SCALES}
\label{sec:prel:ts}

For a general introduction to the calculus on time scales we refer
the reader to the book \cite{bookBohner}.
Here we only give those notions and results needed in the sequel.
As usual,  $\mathbb{R}$, $\mathbb{Z}$,
and $\mathbb{N}$ denote, respectively, the set of real,
integer, and natural numbers.

A {\it time scale} $\mathbb{T}$ is an arbitrary nonempty closed subset
of  $\R$.  Besides standard
cases of $\mathbb{R}$ (continuous time) and $\mathbb{Z}$ (discrete time),
many different models of time are used. For each time scale $\mathbb{T}$
the following operators are used:

\begin{itemize}
\item the {\it forward jump operator} $\sigma:\T \rightarrow \T$,
$\sigma(t):=\inf\{s \in \T:s>t\}$ for $t\neq\sup \T$ and
$\sigma(\sup\T)=\sup\T$ if $\sup\T<+\infty$;

\item the {\it backward jump operator} $\rho:\T \rightarrow \T$,
$\rho(t):=\sup\{s \in \T:s<t\}$ for $t\neq\inf \T$ and
$\rho(\inf\T)=\inf\T$ if $\inf\T>-\infty$;

\item the {\it forward graininess function} $\mu:\T \rightarrow [0,\infty[$,
$\mu(t):=\sigma(t)-t$.
\end{itemize}

\begin{ex}
If $\mathbb{T}=\mathbb{R}$, then for any $t \in \mathbb{R}$,
$\sigma(t)=t=\rho(t)$ and $\mu(t) \equiv 0$.
If $\T=\Z$, then for every $t \in \Z$,
$\sigma(t)=t+1$, $\rho(t)=t-1$ and $\mu(t) \equiv 1$.
\end{ex}

A point $t\in\mathbb{T}$ is called \emph{right-dense},
\emph{right-scattered}, \emph{left-dense} or
\emph{left-scattered} if $\sigma(t)=t$, $\sigma(t)>t$, $\rho(t)=t$,
or $\rho(t)<t$, respectively. We say that $t$ is \emph{isolated}
if $\rho(t)<t<\sigma(t)$, that $t$ is \emph{dense} if $\rho(t)=t=\sigma(t)$.

If $\sup \T$ is finite and left-scattered, we define
$$
\mathbb{T}^\kappa :=\mathbb{T}\setminus \{\sup\T\} \, .
$$
Otherwise, $\mathbb{T}^\kappa :=\mathbb{T}$.

\begin{definition}
Let $f:\T \rightarrow \R$ and $t \in \T^\kappa$. The {\it delta derivative}
of $f$ at $t$ is the real number $f^{\Delta}(t)$ with the property
that given any $\varepsilon > 0$ there is a neighborhood $U$ of $t$
such that
\[|(f(\sigma(t))-f(s))-f^{\Delta}(t)(\sigma(t)-s)| \leq \varepsilon|\sigma(t)-s|\]
for all $s \in U$. We say that $f$ is {\it delta differentiable}
on $\T$ provided $f^{\Delta}(t)$ exists for all $t \in \T^\kappa$.
\end{definition}

We shall often denote $f^\Delta(t)$ by $\frac{\Delta}{\Delta t} f(t)$
if $f$ is a composition of other functions. The delta derivative
of a function $f:\T \rightarrow \R^n$ ($n \in \mathbb{N}$)
is a vector whose components are
delta derivatives of the components of $f$. For $f:\T \rightarrow X$,
where $X$ is an arbitrary set, we define $f^\sigma:=f\circ\sigma$.

For delta differentiable $f$ and $g$, the next formulas hold:

\begin{align*}
f^\sigma(t)&=f(t)+\mu(t)f^\Delta(t) \, ,\\
(fg)^\Delta(t)&=f^\Delta(t)g^\sigma(t)+f(t)g^\Delta(t)\\
&=f^\Delta(t)g(t)+f^\sigma(t)g^\Delta(t).
\end{align*}

\begin{remark}
If $\T=\R$, then $f:\R \rightarrow \R$ is
delta differentiable at $t \in \R$ if and only if $f$ is differentiable in
the ordinary sense at $t$. Then, $f^{\Delta}(t)=\frac{d}{dt}f(t)$.
If $\T=\Z$, then $f:\Z \rightarrow \R$ is always delta differentiable
at every $t \in \Z$ with $f^{\Delta}(t)=f(t+1)-f(t)$.
\end{remark}

Let $a,b \in \mathbb{T}$, $a<b$.
We define the interval $[a,b]_\T$ in $\mathbb{T}$ by
$$[a,b]_\T :=\{ t \in \mathbb{T}: a\leq t\leq b\}.$$
Open intervals and half-open intervals
in $\mathbb{T}$ are defined accordingly.

\begin{theorem}[Corollary~2.9 of \cite{Guseinov:2002}]
\label{monotonia delta}
Let $f:[a,b]_\T\rightarrow\mathbb{R}$ be a continuous
function that has a delta derivative at each point
of $[a,b]_\T^\kappa$. Then $f$ is increasing, decreasing,
non-decreasing, and non-increasing on $[a,b]_\T$ if
$f^{\Delta}(t)>0$, $f^{\Delta}(t)<0$, $f^{\Delta}(t)\geq 0$
and $f^{\Delta}(t)\leq 0$ for all $t \in [a,b]_\T^\kappa$,
respectively.
\end{theorem}

\begin{definition}
A function $F:\mathbb{T}\rightarrow\mathbb{R}$ is called a
\emph{delta antiderivative} of
$f:\mathbb{T}\rightarrow\mathbb{R}$ provided
$$
F^{\Delta}(t)=f(t), \qquad  \forall t \in \mathbb{T}^\kappa.
$$
In this case we define the \emph{delta integral} of $f$ from $a$
to $b$ ($a,b \in \mathbb{T}$) by
\begin{equation*}
\int_{a}^{b}f(t)\Delta t:=F(b)-F(a) \, .
\end{equation*}
\end{definition}

\begin{ex}
If $\mathbb{T}=\mathbb{R}$, then \[\int\limits_{a}^{b}  f(t) \Delta
t=\int\limits_{a}^{b}  f(t) d t,\] where the integral on the
right hand side is the usual Riemann integral.
If $\mathbb{T}=h\mathbb{Z}$, where $h>0$, then
\[\int\limits_{a}^{b}  f(t) \Delta
t=\sum\limits_{k=\frac{a}{h}}^{\frac{b}{h}-1}  h\cdot f(kh),\] for $a<b$.
\end{ex}

In order to present a class of functions that possess a delta
antiderivative, the following definition is introduced:

\begin{definition}
A function $f:\mathbb{T} \to \mathbb{R}$ is called {\emph{rd-continuous}}
if it is continuous at the right-dense points in $\mathbb{T}$
and its left-sided limits exist (finite)
at all left-dense points in $\mathbb{T}$.
A function $f:\mathbb{T} \to \mathbb{R}^n$ is
{\emph{rd-continuous}} if all its components
are rd-continuous.
\end{definition}

We remark that a rd-continuous function defined on a compact interval,
with real values, is bounded. The set of all rd-continuous functions
$f:\mathbb{T} \to \mathbb{R}^n$ is denoted
by $\mathrm{C}_{rd}(\mathbb{T}, \mathbb{R}^n)$,
or simply by $\mathrm{C}_{rd}$. Similarly,
$\mathrm{C}^1_{rd}(\mathbb{T}, \mathbb{R}^n)$ and
$\mathrm{C}^1_{rd}$
will denote the set of functions from $\mathrm{C}_{rd}$
whose delta derivative belongs to $\mathrm{C}_{rd}$.

\begin{theorem}[Theorem~1.74 of \cite{bookBohner}]
\label{antiderivada}
Every rd-continuous function has a delta
antiderivative. In particular, if $a \in \mathbb{T}$,
then the function $F$ defined by
$$
F(t)= \int_{a}^{t}f(\tau)\Delta\tau, \quad t \in \mathbb{T} \, ,
$$
is a delta antiderivative of $f$.
\end{theorem}

The following results will be very useful
in the proof of our main result
(Theorem~\ref{secondEL}).

\begin{theorem}[Theorems~1.93, 1.97, and 1.98 of \cite{bookBohner}]
\label{theorems}
Assume that $\nu:\mathbb{T}\rightarrow\R$ is strictly increasing and
$\widetilde{\mathbb{T}}:=\nu(\mathbb{T})$ is a time scale.
\begin{enumerate}
\item  (Chain rule) Let $\omega: \widetilde{\mathbb{T}}
\rightarrow\R$. If $\nu^\Delta(t)$ and
$\omega^{\widetilde{\Delta}}(\nu (t))$ exist for all
$t \in \mathbb{T}^\kappa$, then
$$(\omega\circ\nu)^\Delta=(\omega^{\widetilde{\Delta}}
\circ\nu)\nu^\Delta.$$
\item  (Derivative of the inverse) The relation
$$ (\nu^{-1})^{\widetilde{\Delta}}(\nu(t))= \frac{1}{\nu^\Delta(t)}$$
holds at points $t \in  \mathbb{T}^\kappa$ where $\nu^\Delta(t) \neq 0.$
\item (Substitution in the integral)  If
$f:\widetilde{\mathbb{T}}\rightarrow\R$ is a $\mathrm{C}_{rd}$ function and
$\nu$ is a $\mathrm{C}^1_{rd}$ function, then for $a,b \in
\mathbb{T}$,
$$
\int_{a}^{b} f(\nu(t))\nu^\Delta(t)\Delta t =
\int_{\nu(a)}^{\nu(b)} f(s)\widetilde{{\Delta}}s.
$$
\end{enumerate}
\end{theorem}

\begin{definition}
We say that $y_{\ast}\in C_{rd}^{1}([a,b]_\T, \mathbb{R}^n)$
is a local minimizer for problem (\ref{problem})
if there exists $\delta > 0$ such that
$$
I[y_{\ast}]\leq I[y]
$$
for all $y \in C_{rd}^{1}([a,b]_\T, \mathbb{R}^n)$
satisfying the boundary conditions
$q(a)=q_a$, $q(b)=q_b$, and
$$
\parallel y - y_{\ast}\parallel :=
\sup_{t \in [a,b]_\T^\kappa}\mid y^{\sigma}(t)-y_{\ast}^{\sigma}(t)\mid
+ \sup_{t \in [a,b]_\T^\kappa}\mid y^{\Delta}(t)-y_{\ast}^{\Delta}(t)\mid < \delta \, ,
$$
where $|\cdot|$ denotes a norm in $\mathbb{R}^n$.
\end{definition}

We recall now the (first) Euler-Lagrange equation
as presented in \cite{BohnerCV04}. As in the introduction,
we use $\partial_i L$
to denote the partial derivative of $L$ with respect
to the $i$-th variable (or group of variables).

\begin{theorem}[Theorem~4.2 of \cite{BohnerCV04}]
\label{Thm:ELts}
If $q$ is a local minimizer of \eqref{problem},
then $q$ satisfies the following Euler-Lagrange equation:
\begin{equation}
\label{eq:el} \frac{\Delta}{\Delta t}\partial_{3}
L\left(t,q^\sigma(t),{q}^\Delta(t)\right) =
\partial_{2} L\left(t,q^\sigma(t),{q}^\Delta(t)\right)
\ , \ t \in [a,b]_\T^{\kappa}.
\end{equation}
\end{theorem}

%-----------------------------------------

\section{MAIN RESULTS}
\label{sec:MR}

The following theorem presents a generalization
to time scales of the second Euler-Lagrange
equation \cite{book:Troutman} (also known as the
DuBois-Reymond equation \cite{Cesari}).

\begin{theorem}\emph{(the second Euler-Lagrange
equation on time scales)}:
\label{secondEL}
If $q\in \mathcal{D}$ is a local minimizer of problem \eqref{problem},
then $q$ satisfies the equation
\begin{equation}
\label{2equationEL}
\frac{\Delta}{\Delta t} \mathcal{H}(t,q^\sigma(t),q^\Delta(t))
=-\partial_1 L(t,q^\sigma(t),q^\Delta(t))
\end{equation}
for all $t\in [a,b]_\T^\kappa$, where
$$
\mathcal{H}(t,u,v)=-L(t,u,v) + \partial_3 L(t,u,v)v
+ \partial_1 L(t,u,v) \mu(t) \, ,
$$
$t\in\T$ and $u,v\in \mathbb{R}^n$.
\end{theorem}

\begin{proof}
Let $q_0\in \mathcal{D}$ be a local minimizer of functional
$I$ in \eqref{problem}.
We will prove that there exists
$c\in\mathbb{R}^n$, $c\neq 0$, that satisfies the condition
\begin{equation} \label{c}
1-c^T q_0^\Delta(t)>0, \quad \forall  t\in[a,b]_\T^\kappa.
\end{equation}
If $q_0^\Delta=0$, then any $c\in\mathbb{R}^n$ satisfies condition \eqref{c}.
Suppose now that $q_0^\Delta\neq 0$. Then there exists some $i=1,2,\ldots, n$
such that $q_{0,i}^\Delta \neq 0$ where we suppose that
$q_0=(q_{0,1}, q_{0,2}, \ldots, q_{0,n})$.
Since $q_{0,i}^\Delta$ is bounded on $[a,b]_\T^\kappa$,
then there exist $m, M \in \mathbb{R}$ such that
$$m \leq q_{0,i}^\Delta (t) \leq M, \quad \forall  t\in[a,b]_\T^\kappa.$$
Let $c:=(c_1, c_2, \ldots, c_n)$ where $c_j=0$ if $j\neq i$.
If $M>0$ we can choose $c_i$ such that $0<c_i<\frac{1}{M}$.
If $M\leq 0$ we can choose $c_i$ such that $\frac{1}{m}<c_i<0$.

The map $S:[a,b]_\T \rightarrow \mathbb{R}$ defined by $$S(t)=t-c^Tq_0(t)$$
is delta differentiable with
$S^\Delta(t)= 1-c^Tq_0^\Delta(t)$ and, by Theorem~\ref{monotonia delta},
$S$ is strictly increasing on $[a,b]_\T$. Note that
$\tilde{\T}=S([a,b]_\T)$ is a new time scale
(because $S$ is continuous and $[a,b]_\T$ is closed).
By $\tilde{\sigma}$ we denote
the forward jump operator and by $\tilde{\Delta}$ we denote
the  delta derivative on $\tilde{\T}$. Let $\tau=S(t)$ and define
$\eta_0(\tau):=q_0(S^{-1}(\tau))$ for $\tau\in \tilde{\T}$.
Note that
\begin{equation*}
t=S^{-1}(\tau) = \tau+c^T\eta_0(\tau)\, ,
\end{equation*}
and
\begin{equation}
\label{eq2}
\eta_0(\tau)=q_0(\tau+c^T\eta_0(\tau)) \, .
\end{equation}
By the chain rule and from \eqref{eq2},
\begin{equation*}
\eta_0^{\tilde{\Delta}}(\tau)
= q_0^\Delta(\tau+c^T\eta_0(\tau))(1+c^T\eta_0^{\tilde{\Delta}}(\tau)),
\end{equation*}
which gives
\begin{equation}
\label{eq5}
q_0^\Delta(t)=\frac{\eta_0^{\tilde{\Delta}}(\tau)}{1
+c^T\eta_0^{\tilde{\Delta}}(\tau)}.
\end{equation}
By the derivative of the inverse applied to $S$ we can conclude that
\begin{equation}
\label{eq6}
\frac{1}{1+c^T\eta_0^{\tilde{\Delta}}(\tau)}=1-c^Tq_0^\Delta(t).
\end{equation}
Note that, since $\widetilde{\sigma}\circ S= S\circ \sigma$, then
\begin{equation}
\label{eq6a}
\begin{split}
\widetilde{\mu}(\tau) &= \widetilde{\sigma}(\tau)-\tau
=\widetilde{\sigma}(S(t))-S(t)\\
&=S^\sigma(t)-S(t)=\mu(t)S^\Delta(t)\\
&=\mu(t)(1-c^Tq_0^\Delta(t))
\end{split}
\end{equation}
and
\begin{equation}
\label{eq3}
\begin{split}
\eta_0(\tilde{\sigma}(\tau)) &= \eta_0(\tilde{\sigma}(S(t)))\\
&= \eta_0(S\circ {\sigma}(t)) = q_0(S^{-1}(S\circ \sigma)(t))\\
&= q_0(\sigma(t)).
\end{split}
\end{equation}

From (\ref{eq5}), (\ref{eq3}), and the substitution in the integral,
\begin{equation*}
\begin{split}
I[q_0] &= \int_a^b L(t,q_0^\sigma(t),q_0^\Delta(t)) \Delta t\\
&=\int_\alpha^\beta \tilde{L}(\tau,\eta_0^{\tilde{\sigma}}(\tau),
\eta_0^{\tilde{\Delta}}(\tau)) \tilde{\Delta} \tau =: \tilde{I}[\eta_0],
\end{split}
\end{equation*}
where
\begin{equation}
\label{eq8}
\tilde{L}(\tau,\nu,\zeta)
=L\left(\tau+c^T\nu-c^T\tilde{\mu}(\tau)\zeta,\nu,\frac{\zeta}{1+c^T\zeta}\right)
\left(1+c^T\zeta\right),
\end{equation}
for $\tau\in\tilde{\mathbb{T}}$, $\nu,\zeta\in\mathbb{R}^n$,
$1+c^T\zeta>0$, $\alpha=S(a)$ and $\beta=S(b)$.
Let
\begin{equation*}
\begin{split}
\mathcal{E}=\{ \eta\ | &\ \eta : \tilde{\T} \rightarrow \R^n,
\ \eta\in \mathrm{C}^1_{rd},\\
&\ \eta(\alpha)=\eta_0(\alpha),\
\eta(\beta)=\eta_0(\beta),\\
&\  1+c^T\eta^{\tilde{\Delta}}(\tau)>0
\text{ for } \tau\in \tilde{\T} \}.
\end{split}
\end{equation*}
We remark that $c$ was chosen so small that
the constraint $1+c^T\eta^{\tilde{\Delta}}(\tau)>0$
is always satisfied for any function $\eta$
in the ``nearby'' of $\eta_0$.
Since $q_0$ is by assumption a local minimizer of $I$ in $\mathcal{D}$,
it follows that $\eta_0$ is a local minimizer of $\tilde{I}$ in $\mathcal{E}$,
so it satisfies the Euler-Lagrange equation (in integral form)
\begin{equation}
\label{eq9}
\partial_3 \tilde{L}(\tau,\eta_0^{\tilde{\sigma}}(\tau),
\eta_0^{\tilde{\Delta}}(\tau))
= \int_\alpha^\tau \partial_2 \tilde{L}(s,\eta_0^{\tilde{\sigma}}(s),
\eta_0^{\tilde{\Delta}}(s))\tilde{\Delta}s + C_1,
\end{equation}
where $C_1$ is a constant vector.
Differentiating \eqref{eq8} we obtain
\begin{equation*}\label{eq10}
\begin{split}
\partial_2 \tilde{L} &\left(\tau,\nu,\zeta\right)\\
=&\partial_1 L\left(\tau+c^T\nu-c^T\tilde{\mu}(\tau)\zeta,\nu,\frac{\zeta}{1+c^T\zeta}\right)
\left(1+c^T\zeta\right)c^T\\
&+\partial_2 L\left(\tau+c^T\nu-c^T\tilde{\mu}(\tau)\zeta,\nu,\frac{\zeta}{1+c^T\zeta}\right)
\left(1+c^T\zeta\right)
\end{split}
\end{equation*}
and
\begin{equation*}
\label{eq11}
\begin{split}
\partial_3 &\tilde{L}\left(\tau,\nu,\zeta\right)\\
&= L\left(\tau+c^T\nu-c^T\tilde{\mu}(\tau)\zeta,\nu,\frac{\zeta}{1+c^T\zeta}\right)c^T\\
&+ \partial_3 L\left(\tau+c^T\nu-c^T\tilde{\mu}(\tau)\zeta,\nu,\frac{\zeta}{1+c^T\zeta}\right)
\left(1+c^T\zeta\right)^{-1}\\
&-\partial_1 L\left(\tau+c^T\nu-c^T\tilde{\mu}(\tau)\zeta,\nu,\frac{\zeta}{1+c^T\zeta}\right)
c^T\tilde{\mu}(\tau)\left(1+c^T\zeta\right).
\end{split}
\end{equation*}
Using \eqref{eq5}, \eqref{eq6},
\eqref{eq6a} and (\ref{eq3}) we obtain
\begin{equation*}
\begin{split}
&\partial_3 \tilde{L}\left(\tau,\eta_0^{\tilde{\sigma}}(\tau),\eta_0^{\tilde{\Delta}}(\tau)\right)\\
&= L\left(\tau+c^T\eta_0^{\tilde{\sigma}}(\tau)
-c^T\eta_0^{\tilde{\Delta}}(\tau)\tilde{\mu}(\tau),\eta_0^{\tilde{\sigma}}(\tau),
\frac{\eta_0^{\tilde{\Delta}}(\tau)}{1+c^T\eta_0^{\tilde{\Delta}}(\tau)}\right)c^T\\
&+ \partial_3 L\left(\tau+c^T\eta_0^{\tilde{\sigma}}(\tau)
-c^T\eta_0^{\tilde{\Delta}}(\tau)\tilde{\mu}(\tau),\eta_0^{\tilde{\sigma}}(\tau),
\frac{\eta_0^{\tilde{\Delta}}(\tau)}{1+c^T\eta_0^{\tilde{\Delta}}(\tau)}\right)\\
&\cdot \left(1+c^T\eta_0^{\tilde{\Delta}}(\tau)\right)^{-1}
- c^T\tilde{\mu}(\tau) \left(1+c^T\eta_0^{\widetilde{\Delta}}(t)\right)\\
&\cdot \partial_1 L\left(\tau+c^T\eta_0^{\tilde{\sigma}}(\tau)
-c^T\eta_0^{\tilde{\Delta}}(\tau)\tilde{\mu}(\tau),\eta_0^{\tilde{\sigma}}(\tau),
\frac{\eta_0^{\tilde{\Delta}}(\tau)}{1+c^T\eta_0^{\tilde{\Delta}}(\tau)}\right)\\
&= c^TL\left(t,q_0^\sigma(t),q_0^\Delta(t)\right)
+ \partial_3 L\left(t,q_0^\sigma(t),q_0^\Delta(t)\right)\left(1-c^Tq_0^\Delta(t)\right)\\
&- c^T\mu(t)\left(1-c^Tq_0^\Delta(t)\right)
\left(1+c^T\eta_0^{\widetilde{\Delta}}(t)\right)
\partial_1 L\left(t,q_0^\sigma(t),q_0^\Delta(t)\right)\\
&= c^TL\left(t,q_0^\sigma(t),q_0^\Delta(t)\right)
+\partial_3 L\left(t,q_0^\sigma(t),q_0^\Delta(t)\right)
\left(1-c^Tq_0^\Delta(t)\right)\\
&- c^T\mu(t)\partial_1 L\left(t,q_0^\sigma(t),q_0^\Delta(t)\right).
\end{split}
\end{equation*}
Note that
\begin{equation*}
\begin{split}
\int_\alpha^\tau
&\partial_2 \tilde{L}\left(s,\eta_0^{\tilde{\sigma}}(s),\eta_0^{\tilde{\Delta}}(s)\right)\tilde{\Delta}s + C_1\\
&=\int_\alpha^\tau c^T \partial_1 L\left(S^{-1}(s),q_0^\sigma(S^{-1}(s)),q_0^\Delta(S^{-1}(s))\right)
(S^{-1})^{\widetilde{\Delta}}(s)\widetilde{\Delta} s\\
& \ \ + \int_\alpha^\tau \partial_2 L\left(S^{-1}(s),q_0^\sigma(S^{-1}(s)),q_0^\Delta(S^{-1}(s))\right)
(S^{-1})^{\widetilde{\Delta}}(s)\widetilde{\Delta} s + C_1\\
&= \int_a^t c^T \partial_1 L\left(s,q_0^\sigma(s),q_0^\Delta(s)\right) \Delta s\\
& \ \ + \int_a^t \partial_2 L\left(s,q_0^\sigma(s),q_0^\Delta(s)\right) \Delta s + C_1 \, .
\end{split}
\end{equation*}
Hence, by the Euler-Lagrange equation (\ref{eq9}),
we may conclude that
\begin{equation*}
\begin{split}
c^T & L\left(t,q_0^\sigma(t),q_0^\Delta(t)\right)
+ \partial_3 L\left(t,q_0^\sigma(t),q_0^\Delta(t)\right)
\left(1-c^Tq_0^\Delta(t)\right)\\
& \ \ - c^T\mu(t)\partial_1 L\left(t,q_0^\sigma(t),q_0^\Delta(t)\right)\\
&=\int_a^t c^T \partial_1 L\left(s,q_0^\sigma(s),q_0^\Delta(s)\right) \Delta s\\
& \ \ + \int_a^t \partial_2 L\left(s,q_0^\sigma(s),q_0^\Delta(s)\right) \Delta s + C_1 .
\end{split}
\end{equation*}
The last equality may be rewritten as
\begin{equation*}
\label{eq13}
\begin{split}
c^T \Bigl[ &
L\left(t,q_0^\sigma(t),q_0^\Delta(t)\right)
- \partial_3 L\left(t,q_0^\sigma(t),q_0^\Delta(t)\right) q_0^\Delta(t)\\
&- \int_a^t \partial_1 L\left(s,q_0^\sigma(s),q_0^\Delta(s)\right) \Delta s
-\partial_1 L\left(t,q_0^\sigma(t),q_0^\Delta(t)\right)\mu(t) \Bigr]\\
= & - \left[\partial_3 L\left(t,q_0^\sigma(t),q_0^\Delta(t)\right)
- \int_a^t  \partial_2 L\left(s,q_0^\sigma(s),q_0^\Delta(s)\right) \Delta s
- C_1\right].
\end{split}
\end{equation*}
Using the Euler-Lagrange equation for $q_0$
we arrive at the intended statement.
\end{proof}

If $\mathbb{T}=\mathbb{R}$, then the equation (\ref{2equationEL})
simplifies due to the fact that $\mu=0$,
and we obtain the classical second Euler-Lagrange equation
(\textrm{cf.}, \textrm{e.g.}, \cite{book:Troutman}):

\begin{corollary}[the second Euler-Lagrange equation]
If $q$ is a local minimizer of the classical
functional \eqref{eq:cp} of the calculus of variations, then
\begin{multline*}
\frac{d}{dt}\left[-L(t,q(t), \dot{q}(t))
+ \partial_3L(t,q(t), \dot{q}(t))\cdot \dot{q}(t))\right]\\
= -\partial_1 L (t,q(t), \dot{q}(t))
\end{multline*}
holds for all $t \in [a,b]$.
\end{corollary}

In the autonomous case, Theorem~\ref{secondEL} gives
an extension of the classical
second Erdmann condition \eqref{eq:cla:2nd:Erd}:

\begin{corollary}[the second Erdmann condition on time scales]
If $q\in \mathcal{D}$ is a local minimizer of the problem
\begin{equation*}
I[q]=\int_a^b L(q^\sigma(t),q^\Delta(t)) \Delta t
\longrightarrow \min_{q\in \mathcal{D}},
\end{equation*}
then $q$ satisfies equation \eqref{eq:cons:Eng}
for all $t\in [a,b]_\T^\kappa$.
\end{corollary}

\begin{ex}
\label{Example1}
Let $\T$ be a time scale with $a,b \in \T$, $a < b$.
Consider problem (\ref{problem}) with $n=1$
and a Lagrangian $L$ given by
$L(t,q^\sigma,q^\Delta)=(q^\Delta)^2$.
The second Euler-Lagrange equation \eqref{2equationEL}
for this problem is
$$
\frac{\Delta}{\Delta t} ((q^\Delta (t))^2)=0,
$$
and the extremal is
$q(t)=ct+k$ with
$$
c=\frac{q_{b}-q_{a}}{b-a} \, ,
\quad  k=\frac{bq_{a}-a q_{b}}{b-a}.
$$
\end{ex}

\begin{ex}
Let $\mathbb{T} = \{0,\frac{1}{8},\frac{1}{4},\frac{3}{8},
\frac{1}{2},\frac{5}{8},\frac{3}{4},\frac{7}{8},1\}$
and consider the following problem on $\mathbb{T}$:
\begin{gather*}
I[q] = \int_0^1 \left[(q^\Delta(t))^2 - 1\right]^2 \Delta t \longrightarrow \min \, ,\\
q(0) = 0 \, , \quad q(1) = 0 \, , \\
q \in C_{rd}^1(\mathbb{T}; \mathbb{R}) \, .
\end{gather*}

The Euler-Lagrange equation \eqref{eq:el} takes the form
\begin{equation}
\label{eq:EL:nwar}
q^\Delta(t) \left[(q^\Delta(t))^2 - 1\right] = \text{const}
\end{equation}
while the second Euler-Lagrange equation \eqref{2equationEL} asserts that
\begin{equation}
\label{eq:2ndEL:nwar}
\left[(q^\Delta(t))^2 - 1\right] \left[1 + 3 (q^\Delta(t))^2\right] = \text{const} \, .
\end{equation}
Let $\tilde{q}(t) = 0$ for all $t \in \mathbb{T} \setminus \left\{\frac{1}{8}, \frac{7}{8}\right\}$,
and $\tilde{q}\left(\frac{1}{8}\right) = \tilde{q}\left(\frac{7}{8}\right) = \frac{1}{8}$.
One has $\tilde{q}^\Delta(0) = \tilde{q}^\Delta\left(\frac{3}{4}\right) = 1$,
$\tilde{q}^\Delta\left(\frac{1}{8}\right) = \tilde{q}^\Delta\left(\frac{7}{8}\right) = -1$,
and $\tilde{q}^\Delta\left(\frac{i}{8}\right) = 0$, $i = 2,3,4,5$.
We see that $\tilde{q}$ is an extremal, \textrm{i.e.}, it satisfies the
Euler-Lagrange equation \eqref{eq:EL:nwar}. However $\tilde{q}$ cannot
be a solution to the problem since it does not satisfy
the second Euler-Lagrange equation \eqref{eq:2ndEL:nwar}.
In fact, any function $q$ satisfying $q^\Delta(t) \in \{-1,0,1\}$, $t \in \mathbb{T}^\kappa$,
is an Euler-Lagrange extremal. Among them, only $q^\Delta (t) = 0$ for all $t \in \mathbb{T}^\kappa$
and those with $q^\Delta(t) = \pm 1$ satisfy our condition \eqref{eq:2ndEL:nwar}.
This example shows a problem for which the Euler-Lagrange equation gives
several candidates which are not the solution to the problem, while
our second Euler-Lagrange equation gives a smaller set of candidates.
Moreover, the candidates obtained from our condition lead us directly
to the explicit solution of the problem. Indeed, the null function and any function $q$ with $q(0) = q(1) = 0$ and
$q^\Delta(t) = \pm 1$, $t \in \mathbb{T}^\kappa$, gives $I[q] = 0$.
They are minimizers because $I[q] \ge 0$ for any function $q \in C_{rd}^1$.
\end{ex}

% ------------------------------

\section{AN APPLICATION: NOETHER'S THEOREM}
\label{sec:appl:NT}

Let $U=\{q \  | \  q:[a,b]_\T \rightarrow \mathbb{R}^n$,
$q \in C^1_{rd}\}$, and consider a one-parameter
family of infinitesimal transformations
\begin{equation}
\label{eq:tinf2}
\begin{cases}
\bar{t} = T(t,q, \epsilon) = t + \epsilon\tau(t,q) + o(\epsilon) \, ,\\
\bar{q} = Q(t,q, \epsilon) = q + \epsilon\xi(t,q) + o(\epsilon) \, ,\\
\end{cases}
\end{equation}
where $\epsilon \in \mathbb{R}$, $\tau:[a,b]_\T \times \mathbb{R}^n\rightarrow\mathbb{R}$,
and $\xi:[a,b]_\T \times \mathbb{R}^n\rightarrow\mathbb{R}$
are delta differentiable functions.

We assume that for every $q\in U$ and every $\epsilon$, the map
$[a,b]\ni t \mapsto \alpha(t):= T(t,q(t), \epsilon)\in\mathbb{R}$ is a
strictly increasing $\mathrm{C}^1_{rd}$ function and its image is
again a time scale with the forward shift operator $\bar{\sigma}$
and the delta derivative $\bar{\Delta}$. We recall that  the
following holds:
\begin{equation*}
\bar{\sigma}\circ\alpha = \alpha \circ \sigma .
\end{equation*}

\begin{definition}
\label{def. invariance}
Functional $I$ in \eqref{problem} is said
to be invariant on $U$ under the family of
transformations \eqref{eq:tinf2} if
$$
\frac{d}{d\epsilon}\left\{L\left(T(t,q(t), \epsilon),
Q^{\sigma}(t,q(t), \epsilon), \frac{Q^{\Delta}}{T^{\Delta}}\right)
T^{\Delta}\right\}\Big|_{\epsilon=0} = 0 ,
$$
where, for simplicity of notation, we omit the arguments of functions
$T^{\Delta}$ and $Q^{\Delta}$:
$T^{\Delta} = T^{\Delta}(t,q(t), \epsilon)$,
$Q^{\Delta} = Q^{\Delta}(t,q(t),\epsilon)$.
\end{definition}

\begin{remark}
Note that the invariance notion presented
in \cite[Definition~5]{NT:ts}
implies Definition~\ref{def. invariance}.
Indeed, for any subinterval $[t_a,t_b]_\T\subseteq [a,b]_\T$,
any $q \in U$, and any $\epsilon$, one has
\begin{equation*}
\begin{split}
\int_{t_a}^{t_b} & L(t,q^{\sigma}(t), q^{\Delta}(t))\Delta t\\
&= \int_{\alpha(t_a)}^{\alpha(t_b)} L(\overline{t},\overline{q}
\circ \overline{\sigma}(\overline{t}),\overline{q}^{\overline{\Delta}}(\overline{t}))
\overline{\Delta}\overline{t}\\
&= \int_{t_a}^{t_b} L\left(\alpha(t),\left(\overline{q}
\circ \overline{\sigma}\circ \alpha\right) (t),
\overline{q}^{\overline{\Delta}}(\alpha(t))\right) \alpha^{\Delta}(t) \Delta t\\
&= \int_{t_a}^{t_b} L\left(\alpha(t),\left(\overline{q}\circ \alpha \circ
\sigma\right)(t), \frac{(\overline{q} \circ \alpha)^{\Delta}(t)}{\alpha^{\Delta}(t)}\right)
\alpha^{\Delta}(t) \Delta t\\
&= \int_{t_a}^{t_b} L\left(T(t,q(t),\epsilon),Q^{\sigma}(t,q(t),\epsilon),
\frac{Q^{\Delta}}{T^{\Delta}}\right)
T^{\Delta} \Delta t \, .
\end{split}
\end{equation*}
From the arbitrariness of $t_a$ and $t_b$ it follows that
$$
L\left(T(t,q(t),\epsilon),Q^{\sigma}(t,q(t),\epsilon),
\frac{Q^{\Delta}}{T^{\Delta}}\right)
T^{\Delta} = L(t,q^{\sigma}(t), q^{\Delta}(t))\, ,
$$
and this implies
$$
\frac{d}{d\epsilon}\left\{L\left(T(t, q(t), \epsilon),
Q^{\sigma}(t, q(t), \epsilon), \frac{Q^{\Delta}}{T^{\Delta}}\right)
T^{\Delta}\right\}\Big|_{\epsilon=0} = 0 .
$$
\end{remark}

\begin{lemma}
\label{thm:invariance}
Functional $I$ in \eqref{problem}
is invariant on $U$ under the family
of transformations \eqref{eq:tinf2} if and only if
\begin{multline*}
\partial_{1}L(t,q^{\sigma}(t), q^{\Delta}(t))\tau(t,q(t))
+ \partial_{2}L(t,q^{\sigma}(t), q^{\Delta}(t))\xi^{\sigma}(t,q(t))\\
+\partial_{3}L(t,q^{\sigma}(t), q^{\Delta}(t))\xi^{\Delta}(t,q(t))
+L(t,q^{\sigma}(t), q^{\Delta}(t))\tau^{\Delta}(t,q(t))\\
-q^{\Delta}(t)\partial_{3}L(t,q^{\sigma}(t), q^{\Delta}(t))\tau^{\Delta}(t,q(t))=0
\end{multline*}
for all $t \in [a,b]_\T^\kappa$ and all $q \in U$,
where
$$
\xi^{\sigma}(t,q(t))=\xi(\sigma(t), q(\sigma(t))) \, ,
\quad
\xi^{\Delta}(t,q(t))=\frac{\Delta}{\Delta t}\xi(t,q(t))\, .
$$
\end{lemma}

\begin{proof}
Since
\begin{equation*}
\begin{gathered}
\frac{\partial T(t,q(t),\epsilon)}{\partial\epsilon}\Big|_{\epsilon=0}
= \tau(t,q(t)) \, ,\\
\frac{\partial Q^{\sigma}(t,q(t),\epsilon)}{\partial\epsilon}\Big|_{\epsilon=0}
= \xi^{\sigma}(t,q(t)) \, ,\\
\frac{\partial}{\partial\epsilon}\left(
\frac{Q^{\Delta}(t,q(t),\epsilon)}{T^{\Delta}(t,q(t),\epsilon)}\right)\Big|_{\epsilon=0}
= \xi^{\Delta}(t,q(t))-q^{\Delta}(t)\tau^{\Delta}(t,q(t)) \, ,\\
\frac{\partial T^{\Delta}(t,q(t),\epsilon)}{\partial\epsilon}\Big|_{\epsilon=0}
= \tau^{\Delta}(t,q(t)) \, ,
\end{gathered}
\end{equation*}
the definition of invariance is equivalent to
\begin{multline*}
\partial_{1}L(t,q^{\sigma}(t), q^{\Delta}(t))\tau(t,q(t))\\
+\partial_{2}L(t,q^{\sigma}(t), q^{\Delta}(t))\xi^{\sigma}(t,q(t))\\
+\partial_{3}L(t,q^{\sigma}(t), q^{\Delta}(t))\left(\xi^{\Delta}(t,q(t))
-q^{\Delta}(t)\tau^{\Delta}(t,q(t))\right)\\
+L(t,q^{\sigma}(t), q^{\Delta}(t))\tau^{\Delta}(t,q(t))=0 \, ,
\end{multline*}
which proves the desired result.
\end{proof}

\begin{ex}
\label{Example2}
For Example~\ref{Example1} one has invariance under the family
of transformations \eqref{eq:tinf2} with $\tau=r$
and $\xi=s$, where $r$ and $s$ are arbitrary constants.
\end{ex}

In order to simplify expressions, we write
$L(t,q^\sigma,q^\Delta)$ instead
of $L(t,q^\sigma(t),q^\Delta(t))$.
Similarly for the partial derivatives of $L$.
We recall that $q$ is an \emph{extremal} to problem
\eqref{problem} if it satisfies
the Euler-Lagrange equation \eqref{eq:el}.

\begin{theorem}[Noether's theorem on time scales]
\label{Noether}
If functional $I$ in \eqref{problem} is invariant on
$U$ in the sense of Definition~\ref{def. invariance}
(\textrm{cf.} Lemma~\ref{thm:invariance}), then
\begin{multline*}
\partial_{3}L(t,q^{\sigma},q^{\Delta})\cdot\xi(t,q)
+ \Bigl[L(t,q^{\sigma},q^{\Delta})
-\partial_{3}L(t,q^{\sigma},q^{\Delta})\cdot q^{\Delta}\\
-\partial_{1}L(t,q^{\sigma},q^{\Delta})\cdot\mu(t)\Bigr]\cdot\tau(t,q)
\end{multline*}
is constant along all the extremals of problem (\ref{problem}).
\end{theorem}

\begin{proof}
We must prove that
\begin{multline*}
\mathcal{C} :=
\frac{\Delta}{\Delta t}\Bigl[
\partial_{3}L(t,q^{\sigma},q^{\Delta})\cdot\xi(t,q)\\
+\Bigl(L(t,q^{\sigma},q^{\Delta})
-\partial_{3}L(t,q^{\sigma},q^{\Delta})\cdot q^{\Delta}\\
-\partial_{1}L(t,q^{\sigma},q^{\Delta})\cdot\mu(t)\Bigr)\cdot\tau(t,q) \Bigr]
\end{multline*}
is equal to zero along all the extremals of problem
(\ref{problem}). We begin noting that
\begin{equation*}
\begin{split}
\mathcal{C} &= \frac{\Delta}{\Delta t}
\partial_{3}L(t,q^{\sigma},q^{\Delta})\cdot\xi^{\sigma}(t,q) +
 \partial_{3}L(t,q^{\sigma},q^{\Delta})\cdot \xi^{\Delta}(t,q)\\
& \quad + \frac{\Delta}{\Delta t}
\Bigl[L(t,q^{\sigma},q^{\Delta})-\partial_{3}L(t,q^{\sigma},q^{\Delta})
\cdot q^{\Delta}\\
& \qquad\quad - \partial_{1}L(t,q^{\sigma},q^{\Delta})\cdot\mu(t)\Bigr]
\cdot\tau^{\sigma}(t,q)\\
& \quad +
\Bigl[L(t,q^{\sigma},q^{\Delta})-\partial_{3}L(t,q^{\sigma},q^{\Delta})\cdot
q^{\Delta}\\
& \qquad\quad -\partial_{1}L(t,q^{\sigma},q^{\Delta})\cdot\mu(t)\Bigr]
\cdot\tau^{\Delta}(t,q) \, .
\end{split}
\end{equation*}
Using the first and second Euler-Lagrange equations \eqref{eq:el}
and \eqref{2equationEL}, respectively, we conclude that
$$
\begin{array}{rcl}
\mathcal{C} & = &
\partial_{2}L(t,q^{\sigma},q^{\Delta})\cdot\xi^{\sigma}(t,q) +
\partial_{3}L(t,q^{\sigma},q^{\Delta})\cdot \xi^{\Delta}(t,q) \\
& &\\
& + & \partial_{1}L(t,q^{\sigma},q^{\Delta}) \cdot\tau^{\sigma}(t,q)\\
& &\\
& + &  L(t,q^{\sigma},q^{\Delta})\cdot \tau^{\Delta}(t,q)
-\partial_{3}L(t,q^{\sigma},q^{\Delta})\cdot q^{\Delta} \cdot \tau^{\Delta}(t,q)\\
& &\\
& - & \partial_{1}L(t,q^{\sigma},q^{\Delta})\cdot\mu(t) \cdot\tau^{\Delta}(t,q) \, .
\end{array}
$$
Since $\tau^{\sigma}(t,q)
=\tau(t,q)+\mu(t)\cdot \tau^{\Delta}(t,q)$,
then
$$
\begin{array}{rcl}
\mathcal{C} & = &
\partial_{2}L(t,q^{\sigma},q^{\Delta})\cdot\xi^{\sigma}(t,q) +
\partial_{3}L(t,q^{\sigma},q^{\Delta})\cdot \xi^{\Delta}(t,q) \\
& &\\
& + & \partial_{1}L(t,q^{\sigma},q^{\Delta}) \cdot \tau(t,q)
+ \partial_{1}L(t,q^{\sigma},q^{\Delta}) \cdot \mu(t) \cdot \tau^{\Delta}(t,q) \\
& &\\
& + &  L(t,q^{\sigma},q^{\Delta})\cdot \tau^{\Delta}(t,q)
-\partial_{3}L(t,q^{\sigma},q^{\Delta})
\cdot q^{\Delta} \cdot \tau^{\Delta}(t,q)\\
& &\\
& - & \partial_{1}L(t,q^{\sigma},q^{\Delta})\cdot\mu(t)
\cdot\tau^{\Delta}(t,q) \, .
\end{array}
$$
Hence,
$$
\begin{array}{rcl}
\mathcal{C} & = &
\partial_{2}L(t,q^{\sigma},q^{\Delta})\cdot\xi^{\sigma}(t,q) +
\partial_{3}L(t,q^{\sigma},q^{\Delta})\cdot \xi^{\Delta}(t,q) \\
& &\\
& + & \partial_{1}L(t,q^{\sigma},q^{\Delta})
\cdot \tau(t,q) + L(t,q^{\sigma},q^{\Delta})\cdot \tau^{\Delta}(t,q)\\
& &\\
& - &\partial_{3}L(t,q^{\sigma},q^{\Delta})\cdot q^{\Delta}
\cdot \tau^{\Delta}(t,q) \, .
\end{array}
$$
Using Lemma~\ref{thm:invariance} we arrive
at the intended conclusion.
\end{proof}

If $\mathbb{T}=\mathbb{R}$, then $\mu=0$ and Theorem~\ref{Noether}
reduces to the classical Noether's theorem
(\textrm{cf.}, \textrm{e.g.}, \cite{book:Logan}):

\begin{corollary}[Noether's theorem]
If the classical fundamental functional
of the calculus of variations \eqref{eq:cp}
is invariant, then
$$
\partial_{3}L(t,q,\dot{q})\cdot\xi(t,q)
+ \left[L(t,q,\dot{q})-\partial_{3}L(t,q,\dot{q}) \cdot \dot{q}
\right]\cdot\tau(t,q)
$$
is constant along all the extremals of the problem.
\end{corollary}

\begin{ex}
\label{Example3}
For the problem of Example~\ref{Example2} one has from
Theorem~\ref{Noether} that
\begin{equation}
\label{cons:law:ex}
2sq^\Delta - r(q^\Delta)^2=const
\end{equation}
along the extremals $q$ of the problem. This is indeed
true: from Example~\ref{Example1} we know that the extremals
have the form $q(t)=ct+k$ for some constants $c, k \in
\mathbb{R}$; thus, the conservation law
\eqref{cons:law:ex} takes the form
$2sc-rc^2=const$.
\end{ex}

% ------------------------------

\section{CONCLUSION AND FUTURE WORK}

In this paper we obtain a
second Euler-Lagrange equation
and a second Erdmann condition
for the problem of the calculus of variations
on time scales. Since both necessary optimality conditions are
important and extremely useful results in the calculus of variations
and optimal control when $\T = \R$,
we claim that the present results
are also useful for the development of the
recent theory of the calculus of variations
on time scales \cite{Tangier-with-Rui}.
As pointed out to us by Richard Vinter,
our second Euler-Lagrange equation
in the time scales setting seems to be useful
in a framework for studying the asymptotics
of time discretization.

As an example of application of our main results,
we give a simpler and more elegant proof to the
Noether symmetry theorem on time scales
obtained in 2008 \cite{NT:ts},
which allows to obtain conserved quantities
along the extremals of the problems.
Standard Noetherian constants of motion are violated due to the presence
of a new term that depends on the graininess $\mu(t)$ of the time scale,
while in the classical context $\mu(t) \equiv 0$.
The importance of Noether's conservation laws
in the calculus of variations,
optimal control theory, and its applications
in engineering, are well recognized
\cite{Terence,NonlinearDyn,Paulo,Torres02}.
Their role on the general context
of optimal control on time scales is an entirely open area of research.
In particular, it would be interesting to investigate the techniques
of \cite{Torres02,Torres04} with the recent higher-order Euler-Lagrange
equations on time scales \cite{comRui:TS:Lisboa07,NataliaHigherOrderNabla}
for a possible extension of Theorem~\ref{Noether} to variational problems
on time scales with higher-order delta or nabla derivatives.
The question of obtaining conserved quantities
along the extremals of higher-order problems of the calculus of variations
on time scales remains an interesting open question.

% ------------------------------

\section{ACKNOWLEDGMENTS}

Zbigniew Bartosiewicz was supported by Bia{\l}ystok University of Technology
grant S/WI/1/08; Nat\'{a}lia Martins and Delfim F. M. Torres by the R\&D
unit ``Centre for Research in Optimization and Control'' (CEOC)
of the University of Aveiro, cofinanced by the European Community Fund
FEDER/POCI 2010.

% ------------------------------

%-----------------------------------------

\end{document}